\documentclass[12pt]{article}
\usepackage[english]{babel}
\usepackage{amsthm,amsfonts,amssymb,amsmath}
\usepackage{graphicx}

\newcommand{\R}{\mathbb{R}}
\newcommand{\Q}{\mathbb{Q}}
\newcommand{\C}{\mathbb{C}}

\newcommand{\Z}{\mathbb{Z}}

\newcommand{\es}[1]{\begin{equation}\begin{split}#1\end{split}\end{equation}}
\newcommand{\est}[1]{\begin{equation*}\begin{split}#1\end{split}\end{equation*}}

\newcommand{\s}{\mathcal{S}}

\newcommand{\tn}[1]{\textnormal{#1}}

\newtheorem*{teo*}{Theorem}
\newtheorem{theo}{Theorem}

\newtheorem{lemma}{Lemma}

\newtheorem*{rem*}{Remark}
\newcommand{\pr}[1]{\left( #1\right)}
\newcommand{\pg}[1]{\left\{ #1\right\}}
\newcommand{\pmd}[1]{\left| #1\right|}

\newcommand{\e}[1]{\textnormal{e}\pr{ #1}}

\begin{document}

\author{Sandro Bettin and Brian Conrey}
\title{Period functions and cotangent sums}
\date{}

\maketitle

\begin{abstract}
We investigate the period function of $\sum_{n=1}^\infty\sigma_a(n)\e{nz}$, showing it can be analytically continued to $|\arg z|<\pi$ and studying its Taylor series. We use these results to give a simple proof of the Voronoi formula and to prove an exact formula for the second moment of the Riemann zeta function. Moreover, we introduce a family of cotangent sums, functions defined over the rationals, that generalize the Dedekind sum and share with it the property of satisfying a reciprocity formula. 
\end{abstract}

\section{Introduction}
In the well-known theory of period polynomials one constructs a vector space of polynomials
associated with a vector space of modular forms. The Hecke operators act on each space and have the same eigenvalues. Thus, either vector space
produces the usual degree 2 L-series associated with holomorphic modular forms. In 2001
 Lewis and Zagier extended this theory and defined spaces of period functions associated to non-holomorphic modular forms, i.e. to Maass forms and real analytic Eisenstein series. Period functions are real analytic functions $\psi(x)$ which satisfy three term relations
\begin{equation} \label{eqn:three} \psi(x)=\psi(x+1)+(x+1)^{-2s} \psi\left( \frac{x}{1+x}\right)
\end{equation}
where $s=1/2+it$. The period functions for Maass forms are characterized by (\ref{eqn:three}) together with the growth conditions $\psi(x)=o(1/x)$ as $x\to 0^+$ and $\psi(x)=o(1)$ as $x\to \infty$; for these $s=1/2+ir$ where $1/4+r^2$ is the eigenvalue of the
Laplacian associated with a Maass form. For Eisenstein series, the $o$'s in the above growth conditions are replaced by $O$'s if $t\neq0$ and by $O\pr{\frac1{x|\log x|}}$ and $O(\log x)$ if $t=0$. They show that $\psi$, which is initially defined only in the upper half plane, actually has an analytic continuation to all of $\mathbb C$ apart from the negative real axis. 

To each period function is also associated a periodic and holomorphic function $f$ on the upper half plane,
\est{
f(z)=\psi(z)+z^{-2s}\psi\pr{-\frac1{z}}.
}

In this paper we focus on the case of real analytic Eisenstein series. For these, the periodic function $f$ turns out to be essentially \est{
\sum_{n=1}^\infty\sigma_{2s-1}(n)\e{nz},
}
where, as usual, $\sigma_a(n):=\sum_{d|n}d^a$ indicates the sum of the $a$-th power of the divisors of $n$ and $\e{z}:=e^{2\pi i z}$. We interpret Lewis and Zagier's results directly in terms of this function, obtaining a better understanding of the Taylor series of the associated period function. 
It turns out that the case $s=1/2$, i.e. $t=0$ is especially useful. In this case the arithmetic part of the $n$-th Fourier coefficient is $d(n)$, the number of divisors of $n$. 

There are several nice applications that are consequences of the analytic continuation of the
associated period function, i.e. they are consequences of the surprising fact that the function
$$\sum_{n=1}^\infty d(n) e(nz) -\frac{1}{z} \sum_{n=1}^\infty d(n)e(-n/z),$$
which apparently only makes sense when the imaginary part of $z$ is positive, actually
has an analytic continuation to ${\mathbb C}'$ the split complex plane (the complex with the negative real axis removed).
First, we obtain a new formula for the weighted mean square of the Riemann zeta-function on the critical line:
$$\int_0^\infty |\zeta(1/2+it)|^2 e^{-\delta t} ~ dt.$$
Previously, the best formula for this quantity was a main term plus an asymptotic, but not convergent, series of powers of $\delta$,
 each term an order of magnitude better than the previous as $\delta\to 0^+$.
 Our formula gives an asymptotic series which is also convergent.
 The situation is  somewhat analogous to the situation of the partition function
 $p(n)$. Hardy and Ramanujan found an asymptotic series for $p(n)$ and subsequently Rademacher gave a series which was both asymptotic and convergent. In both the partition case and our case, the exact formula allows for the computation of the sought quantity to {\it any}
 desired degree of precision, whereas an asymptotic series has limits to its precision.
Of course, an extra feature of $p(n)$, that is not present in our situation, is that since $p(n)$  is an integer it is known exactly once it is known to a precision of $0.5$. However, our formula does have the extra surprising feature that
the time required to calculate our desired mean square  is basically independent of $\delta$, apart from the intrinsic difficulty
of the extra work required just to write down a high precision number $\delta$.

A second application proves a surprising reciprocity formula for the Vasyunin sum, which is a cotangent sum that appears in the Nyman-Beurling criterion for the Riemann Hypothesis. Specifically,  the Vasyunin sum appears as part of the exact formula for the twisted mean-square of the Riemann zeta-function on the critical line:
  $$ \int_0^\infty |\zeta(1/2+it)|^2 (h/k)^{it} \frac{dt}{\tfrac 14 +t^2}.$$
  The fact that there is a reciprocity formula for the Vasyunin sum is a non-obvious symmetry relating this integral for $h/k$ and the integral for  $\overline{h}/k$ where $h\overline{h}\equiv 1 \bmod k$.  It is not apparent from this integral that there should be such a  relationship;  our formula reveals a hidden structure. 

The reciprocity formula is most simply stated in terms of the function
$$ c_0(h/k)=-\sum_{m=1}^{k-1}\frac m k \cot \frac {\pi m h}{k}$$
defined initially for non-zero rational numbers $h/k$ where $h$ and $k$ are integers with $(h,k)=1$ and $k>0$. The reciprocity formula can be simply stated as, ``The function
   $$c_0\pr{\frac hk}+\frac kh c_0\pr{\frac kh}-\frac1{\pi h}$$
extends from its initial definition on rationals $x=h/k$ to an (explicit) analytic function on the complex plane with the negative real axis deleted.'' This is nearly an example of what Zagier calls a ``quantum modular form.'' We proved this reciprocity formula in~\cite{BC}; in this paper, we generalize it to a family of ``cotangent sums'', containing both $c_0$ and the Dedekind sum.

These (imperfect) quantum modular forms are analogous to the ``Quantum Maass Forms'' studied by Bruggeman in~\cite{Br}, the former being associated to Eisenstein series and the latter to Maass forms. The main difference between these two classes of quantum forms comes from the fact that the $L$-functions associated to Maass forms are entire, while for Eisenstein series the associated $L$-functions are not, since they are products of two shifted Riemann zeta functions. This translates into Quantum Maass Forms being quantum modular forms in the strict sense, whereas the reciprocity formulas for the cotangent sums contain a non-smooth correction term.

As a third application, we give a generalization of the classical Voronoi summation formula, which is a formula for $\sum_{n=1}^\infty d(n) f(n)$ where $f(n)$ is a smooth rapidly decaying function. The usual formula proceeds from
$$\sum_{n=1}^\infty d(n) f(n)=\frac{1}{2\pi i} \int_{(2)} \zeta(s)^2 \tilde{f}(s) ~ds$$
where
$$\tilde{f}(s)=\int_0^\infty f(x)x^{-s}~dx.$$
One obtains the formula by moving the path of integration to the left to $\Re s = -1$, say, and  then using the functional equation
$$\zeta(s) = \chi(s)\zeta(1-s)$$
of $\zeta(s)$. Here, as usual,
$$ \chi(s)=2(2\pi)^{s-1}\Gamma(1-s). $$
In this way one obtains a leading term
$$ \int_0^\infty f(u)(\log u +2\gamma)~du,$$
from the pole of $\zeta(s)$ at $s=1$, plus another term
$$\sum_{n=1}^\infty d(n) \hat{f}(n)$$
where $\hat{f}(u)$ is a kind of Fourier-Bessel transform of $f$; specifically,
$$\hat{f}(u) =\frac{1}{2\pi i}\int_{(-1)} \chi(s)^2 u^{s-1} \tilde{f}(s)~ds=\int_0^\infty f(t) C(2 \pi \sqrt{tu}) ~ dt$$
with $C(z)=4 K_0(2z) -2 \pi Y_0(2z)$ where $K$ and $Y$ are the usual Bessel functions. By contrast, the period relation implies, for example, that for $0<\delta<\pi$ and $z=1-e^{-i\delta}$
\es{ \label{eqn:vor}
\sum_{n=1}^\infty d(n)\e{n z} =&\, \frac14+2\frac{\log\pr{-2\pi iz}-\gamma}{2\pi i z}+\frac1{z}\sum_{n=1}^\infty d(n) \e{\frac{-n }{z} }+ \sum_{n=1}^ \infty c_n e^{-i n \delta}
}
where $c_n \ll e^{-2\sqrt{\pi n}}$. This is a useful formula which cannot be readily extracted from the Voronoi formula.
In fact, the Voronoi formula is actually an easy consequence of the formula \eqref{eqn:vor}.
In section~\ref{VF} we give some other applications of this extended Voronoi's formula.

The theory and applications described above are for the period function associated with the Eisenstein series with $s=1/2$. In this paper we work in a slightly more general setting with $s=a$, an arbitrary complex number.
The circle of ideas presented here have other applications and further generalizations, for example to exact formulae for averages
of Dirichlet L-functions, which will be explored in future work.

\section{Statement of results}
For $a\in\C$ and $\Im(z)>0$, consider
\est{
\s_a\pr{z}:=\sum_{n=1}^\infty\sigma_a(n)\e{nz}.
}
For $a=2k+1$, $k\in\Z_{\geq1}$, $\s_a(z)$ is essentially the Eisenstein series of weight $2k+2$,
\est{
E_{a+1}(z) = 1+ \frac {2}{\zeta(-a)}\s_{a}\pr{z},
}
for which the well known modularity property
\est{
E_{2k}(z) -\frac{1}{z^{2k}} E_{2k}\pr{-\frac1z}=0
}
holds when $k\ge2$. For other values of $a$ this equality is no longer true, but the period function
\es{\label{rf}
\psi_{a}\pr{z}:=&\,E_{a+1}(z) -\frac{1}{z^{a+1}} E_{a+1}\pr{-\frac1z}\\
}
still has some remarkable properties.
\begin{theo}\label{tb}
Let $\Im(z)>0$ and $a\in\C$. Then $\psi_a(z)$ satisfies the three term relation
\es{\label{3tr}
\psi_a(z)-\psi_a(z+1)=\frac1{(z+1)^{1+a}}\psi_a\pr{\frac{z}{z+1}}
}
and extends to an analytic function on $\C':=\C\setminus{\R_{\leq0}}$ via the representation
\est{
\psi_a(z)=\frac i{\pi z}\frac{\zeta(1-a)}{\zeta(-a)}-i\frac1{z^{1+a}}\cot\frac{\pi a}2+i\frac{g_a(z)}{\zeta(-a)},\\
}
where
\es{\label{g_a}
g_a(z):=&-2\sum_{1\leq n\leq M}(-1)^{n}\frac{B_{2n}}{(2n)!}\zeta(1-2n-a)(2\pi z)^{2n-1}+\\
&+\frac1{\pi i}\int_{\pr{-\frac12-2M}}\zeta(s)\zeta(s-a)\Gamma(s)\frac{\cos\frac{\pi a}2}{\sin\frac {\pi \pr{s-a}}2}(2\pi z)^{-s}\,ds,
}
and $M$ is any integer greater or equal to $-\frac12\min\pr{0,{\Re(a)}}$.
\end{theo}
Here and throughout the  paper equalities are to be interpreted as identities between meromorphic functions in $a$. In particular, taking the limit $a\rightarrow 0^+$, we have
\est{
\psi_0(z)&=-2\frac{\log2\pi z-\gamma}{\pi iz}-2ig_0(z),\\
g_0(z)&=\frac1{\pi i}\int_{\pr{-\frac12}}\zeta(s)^2\frac{\Gamma(s)}{\sin\frac {\pi s}2}(2\pi z)^{-s}\,ds=\frac1{\pi i}\int_{\pr{-\frac12}}\frac{\zeta(s)\zeta(1-s)}{\sin\pi s} z^{-s}\,ds.
}

Theorem~\ref{tb} is essentially a reformulation of Lewis and Zagier's results for the noncuspidal case in~\cite{LZ} and can be seen as a starting point for their theory of period functions. 

For ease of reference, note that~\eqref{rf} can be rewritten in terms of $\s_a$ and $g_a$ as
\es{\label{rfs}
\s_a(z)&-\frac1{z^{a+1}}\s_a\pr{-\frac1z}=\\
&=i\frac {\zeta(1-a)}{2\pi z}-\frac{\zeta(-a)}2
+\frac{e^{\frac{\pi i(a+1)} 2}\zeta(a+1)\Gamma(a+1)}{(2\pi z)^{a+1}}+\frac{i}2g_a(z).
}

Another important feature of the function $\psi_a\pr{z}$ comes from the properties of its Taylor series. For example, in the case $a=0$ one has
\est{
\frac{\pi i}{2} (1+z)\psi_{0}(1+z)=-1-\frac z2+\sum_{m=2}^{\infty}a_m(-z)^m,
}
with
\est{
a_m&:=\frac1{n(n+1)}+2b_n+2\sum_{j=0}^{n-2}{\binom {n-1}j}b_{j+2},\\
b_n&:=\frac{\zeta\pr{n}B_n}{n}
}
and where $B_{2n}$ denotes the $2n$-th Bernoulli number. In particular, the values $a_{m}$ are rational polynomials in $\pi^2$. The terms involved in the definition of $a_m$ are extremely large, since
\est{
b_{2n}\sim\frac{B_{2n}}{2n}\sim(-1)^{n+1}2 \sqrt{\frac \pi n} \left(\frac{n}{ \pi e} \right)^{2n}
}
as $n\rightarrow\infty$, though there is a lot of cancellation; for example, for $m=20$ one has
\est{
a_m=&\,\frac1{420}+\frac{\pi^2}{36}-\frac{19\,\pi^4}{600}+\frac{646\,\pi^6}{19845}-\frac{323\,\pi^8}{1500}+\frac{4199\,\pi^{10}}{343035}+\\
&-\frac{154226363\,\pi^{12}}{36569373750}+\frac{1292\,\pi^{14}}{1403325}-\frac{248571091\,\pi^{16}}{2170943775000}+\\
&+\frac{1924313689\,\pi^{18}}{288905366499750}-\frac{30489001321\,\pi^{20}}{252669361772953125}\\
=&\,0.0499998087\dots
}
Notice how close this number is to $\frac1{20}$; this observation can be made for all $m$ and in fact in~\cite{BC} we proved that
\est{
a_m-\frac{1}{m}\sim{2^\frac 54\pi^\frac34}\frac{e^{-2\sqrt{\pi m}}}{m^\frac34}\pr{\sin\pr{2\sqrt{\pi m}+\frac 38\pi}+O\pr{\frac1{\sqrt m}}}.
}

In this paper we show that similar results hold for the Taylor series at any point $\tau$ in the half plane $\Re(\tau)>0$ and for any $a\in\C$. We give a proof in the following theorem, using $g_a$ instead of $\psi_a$ to simplify slightly the resulting formulae.

\begin{theo}\label{tb2}
Let $\Re(\tau)>0$ and for $|z|<|\tau|$, let
\est{
g_{a}(\tau+z):=\sum_{m=0}^{\infty}\frac{g_{a}^{(m)}(\tau) }{m!}z^m
}
be the Taylor series of $g_a(z)$ around $\tau$. Then
\es{\label{gex}
\frac{g_{a}^{(m)}(1) }{m!}&=-\sum_{\substack{2n-1+k=m,\\ n, k\geq1}}(-1)^{n+m}B_{2n}\zeta(1-2n-a)\frac{\Gamma(2n+a+k)}{\Gamma(2n+a)k!(2n)!}2(2\pi )^{2n-1}+\\
&\quad+(-1)^m\cot\frac{\pi a}2\zeta(-a)\frac{\Gamma(1+a+m)}{\Gamma(1+a)m!}+\\
&\quad+(-1)^m\pr{\frac{\Gamma(1+a+m)}{\Gamma(a)(m+1)!}-1}\frac{\zeta(1-a)}{\pi},
}
and in particular  if $a\in\Z_{\leq0}$, $(a,m)\neq (0,0)$, then $\pi g_{a}^{(m)}(1)$ is a rational polynomial in $\pi^2$. Moreover,
\es{\label{aseq}
\frac{g_{a}^{(m)}(\tau) }{m!}
&=  \cos	\pr{\frac{\pi a}2} \frac{2^{\frac74-\frac a2}}{\pi^{\frac34+\frac a2}} \frac{e^{-2\sqrt {\pi\tau m}}}{m^{\frac14 -\frac a2}\tau^{m+\frac34+\frac a2}}\times\\
&\quad\times \pr{\cos\pr{2\sqrt{\pi\tau m}-\frac\pi8\pr{2a-1}+(\tau +m)\pi}+O_{\tau,a}\pr{\frac1{\sqrt m}}},
}
as $m\rightarrow\infty$.
\end{theo}

Some of the ideas used in the proofs of Theorem~\ref{tb} and~\ref{tb2} can be easily generalized to a more general setting. For example, let $F(s)$ be a meromorphic function on $1-\omega\leq\Re(s)\leq\omega$ for some $1<\omega<2$ with no poles on the boundary and assume $|F(\sigma+it)|\ll_{\sigma} e^{\pr{\frac\pi2 -\eta }|t|}$ for some $\eta>0$. Let
\es{\label{Wpm}
W_{+}\pr{z}:=&\,\frac1{2\pi i}\int_{\pr{\omega}} F\pr{s}\Gamma(s)(-2\pi i z)^{-s}\,ds,\\
W_{-}\pr{z}:=&\,\frac1{2\pi i}\int_{\pr{\omega}} F\pr{1-s}\Gamma(s)(-2\pi i z)^{-s}\,ds,\\
}
for $\frac{\pi}2-\eta<\arg z<\frac\pi 2+\eta$. (Notice that these functions are essentially convolutions of the exponential function and the Mellin transform of $F(s)$.) Then we have
\es{\label{fbF}
\sum_{n=1}^\infty d(n)W_{+}\pr{nz}&-\frac1z\sum_{n=1}^\infty d(n)W_{-}\pr{-\frac nz}=R(z)+k(z),\\
}
where $R(z)$ is the sum of the residues of $F\pr{s}\Gamma(s)\zeta(s)^2(-2\pi i z)^{-s}$ between $1-\omega$ and $\omega$, and
\est{
k(z):=&\,\frac1{2\pi }\int_{\pr{1-\omega}}F(s)\frac{\zeta(s)\zeta(1-s)}{\sin\pi s}z^{-s}\,ds
}
is holomorphic on $|\arg(z)|<\frac{\pi }2+\eta$. Moreover, if we assume that $F(s)$ is holomorphic on $\Re(s)<1-\omega$, then it follows that the Taylor series of $k(z)$ converges very fast,
\est{
\frac{k^{\pr n}(\tau)}{n!}\ll n^{-B}|\tau|^{-n}
}
for any $B>0$ and $\tau$ such that $|\arg\tau|<\eta$. Also, $W_{-}\pr{z}$ decays faster than any power of $z$ at infinity and so the second sum in~\eqref{fbF} is rapidly convergent and is very small if we let $z$ go to zero in $|\arg z|<\eta$. In section~\ref{VF} we will give an explicit example; a subsequent paper will elaborate on this.
 
The Voronoi summation formula is an important tool in analytic number theory; in its simplest form, it states that, if $f(u)$ is a smooth function of compact support, then
\es{\label{vor}
\sum_{n=1}^{\infty}d(n)f(n)=\sum_{n=1}^{\infty}d(n)\hat f (n)+\int_0^\infty f(t)\pr{\log t+2\gamma}\,dt+\frac{f(0)}4,
}
where
\est{
\hat f(x):=4\int_{0}^\infty f(t)\pr{K_0\pr{4\pi\sqrt{tx}}-\frac\pi 2Y_0\pr{4\pi\sqrt{tx}}}\,dt.
}
This formula can be deduced from~\eqref{fbF} (or also directly from~\eqref{rfs}) as a very easy corollary. Actually, Voronoi's formula can be interpreted as a version of the formula~\eqref{rfs} confined to the positive real axis. If we get rid of this limitation and we use directly the period formula~\eqref{rfs}, we are able to obtain interesting results also for weight functions of the shape $f(u)=e^{-\delta u}$, for which the Voronoi summation formula fails to give a useful formula. (Try it!)
Thus, we have a generalization of Voronoi's formula.

The use of a weight function of the shape $e^{-\delta u}$ is fundamental to investigate the smoothly weighted second moment of the Riemann zeta function,
\est{
L_{2k}(\delta):=\int_0^\infty\pmd{\zeta\pr{\frac12+it}}^{2k}e^{-\delta t}\,dt,
}
in the case $k=1$. These integrals play a major role in the theory of the Riemann zeta function and getting good upper bounds on their growth as $\delta\rightarrow0^+$ would imply the Lindel\"of hypothesis. Unfortunately, the only two value of $k$ for which the asymptotics are known are $k=1$ (Hardy and Littlewood,~\cite{HL}) and $k=2$ (Ingham,~\cite{In}); for other values we have just conjectures (see~\cite{CGh},~\cite{CGo} and~\cite{KS}). For $k=1$, it is easy to see that the smooth moment is strictly related to the sum $\s_0\pr{-e^{-i\delta}}$ and, from this, it is easy to deduce an asymptotic expansion for $L_{2k}(\delta)$. This classical asymptotic series is not convergent. Here we replace the series
 by two series, each of which are absolutely convergent asymptotic series. (See also Motohashi \cite{MOT}). The following theorem provides a new exact formula for $L_{1}(\delta)$, by applying Theorem~\ref{tb} and~\ref{tb2} to $\s_0\pr{-e^{-i\delta}}$.
\begin{theo}\label{tr}
For $0<\Re(\delta)<\pi$, we have
\begin{equation*}
\begin{split}
L_1(\delta)=&\frac{\gamma-\log2\pi\delta}{2\sin\frac\delta2}+\frac{\pi i}{\sin\frac\delta2}\s_0\left(\frac {-1}{1-e^{-i\delta}}\right)+h(\delta)+k(\delta),\\
\end{split}
\end{equation*}
where $k(\delta)$ is analytic in $\pmd{\Re(\delta)}<\pi$ and $h(\delta)$ is $C^\infty$ in $\R$ and holomorphic in
\est{
\C'':=\C\setminus\left\{x+iy\in\C\mid x\in2\pi\Z,\ y\geq0\right\}.
}
Moreover, $h(0)=0$ and, if $\Im(\delta)\leq0$,
\est{
h(\delta)=i\sum_{n\geq0} h_{n} e^{-i\pr{n+\frac12}\delta},
}
with
\est{
h_n=  2^\frac74\pi^{\frac14}\frac{e^{-2\sqrt{\pi n}}}{n^\frac14}\sin\pr{2\sqrt{\pi n}+\frac {5\pi}{8}}+O\pr{\frac{e^{-2\sqrt{\pi n}}}{n^\frac34}},
}
as $n\rightarrow\infty.$ 
\end{theo}

The most remarkable aspect of this theorem lies in the fact that the arithmetic sum $\s_0\left(\frac {-1}{1-e^{-i\delta}}\right)$ decays exponentially fast for $\delta\rightarrow0^+$, while the Fourier series $h(\delta)$ is very rapidly convergent. Moreover, Theorem~\ref{tr} implies that $L_1(\delta)$ can be evaluated to any given precision in a time which is independent of $\delta.$

For a rational number $\frac hk$, $(h,k)=1$, $k>0$, define
\est{
c_0\pr{\frac hk}=-\sum_{m=1}^{k-1}\frac mk\cot\pr{\frac{\pi  mh}{k}}.
}
The value of $c_0\pr{\frac hk}$ is an algebraic number, i.e. $c_0:\Q\rightarrow\overline\Q$, and, more precisely, $c_\ell\pr{\frac hk}$ is contained in the cyclotomic field containing all roots of
unity. Moreover, $c_0$ is odd and is periodic of period $1$.
{
\begin{figure}[ht]
\begin{minipage}[b]{0.5\linewidth}
\centering
\includegraphics[height=120pt, width=190pt]{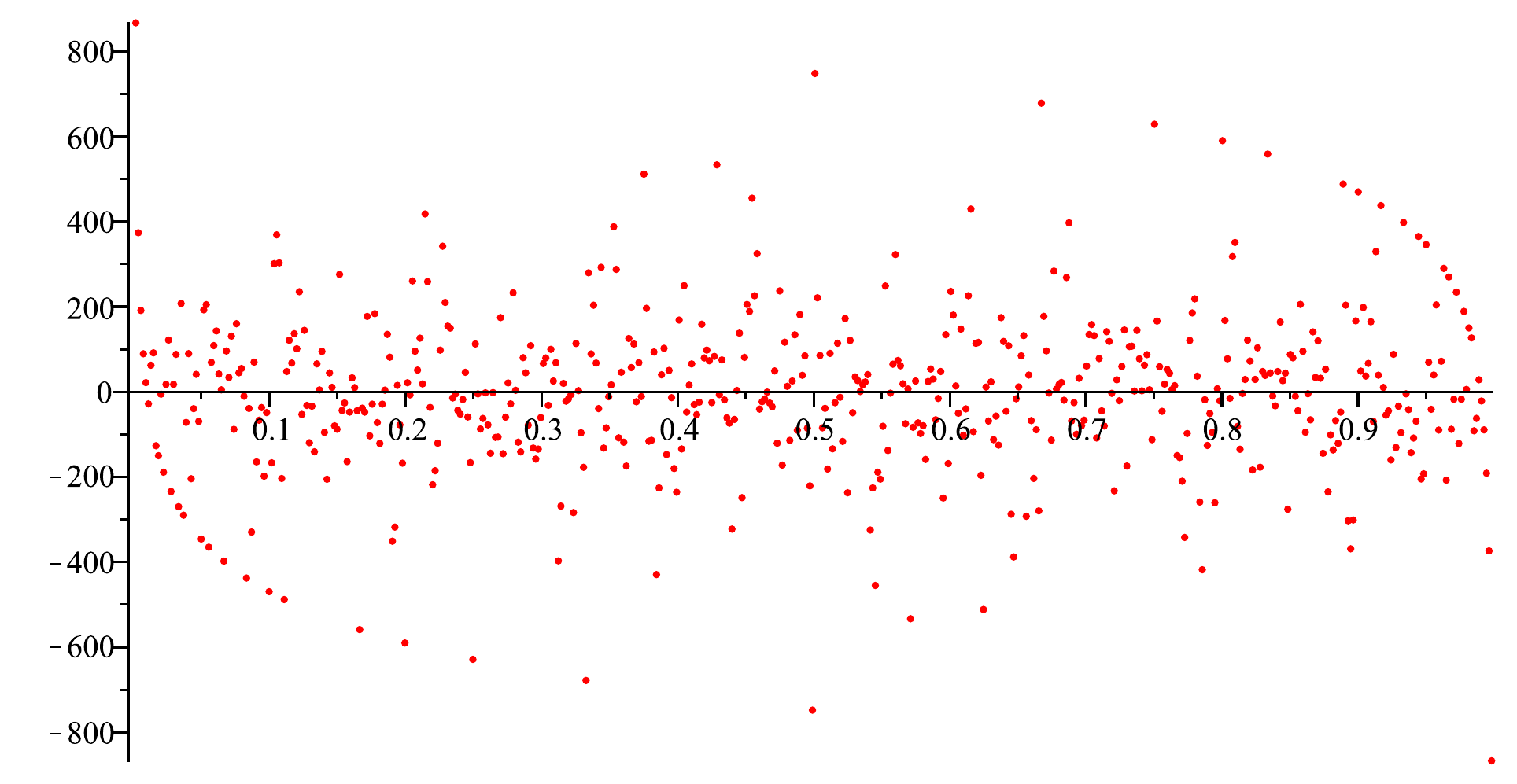}
\caption{Graph of $c_0\pr{\frac hk}$ for $1\leq h<k=541$.}
\label{fig:F1}
\end{minipage} 
\hspace{0.5cm}
\begin{minipage}[b]{0.5\linewidth}
\centering
\includegraphics[width=190pt]{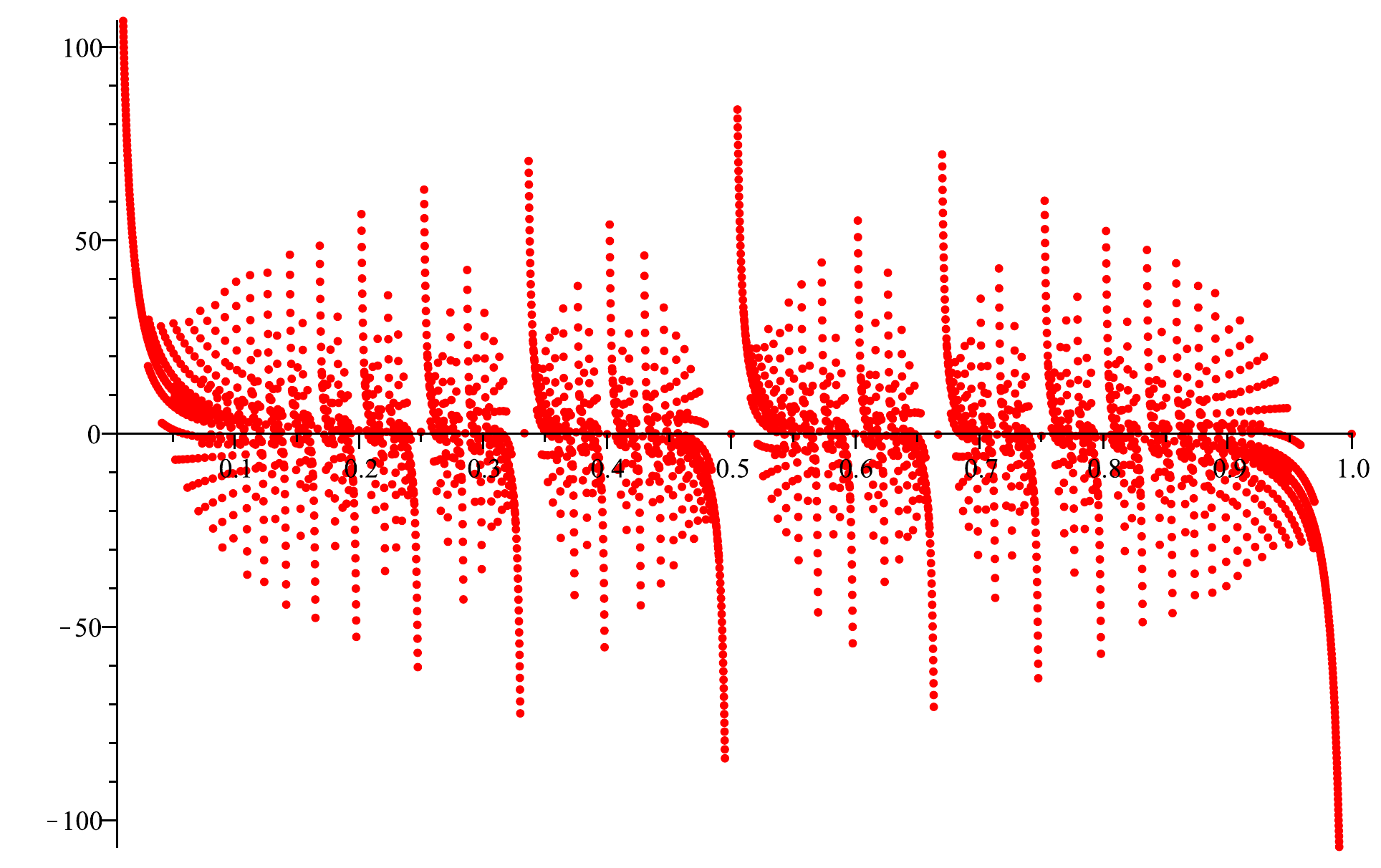}
\caption{Graph of $c_0\pr{\frac hk}$ for $1\leq h\leq k\leq 100$, $(h,k)=1$.}
\label{fig:F2}
\end{minipage}
\end{figure}
}
The cotangent sum $c_0\pr{\frac hk}$ arises in analytic number theory in the value at $s=0$,
\es{\label{dc}
D\pr{0,\frac hk}=\frac 14 +\frac i2 c_{0}\pr{\frac hk},
}
of the Estermann function, defined for $\Re(s)>1$ by
\est{
D\pr{s,\frac hk}:=\sum_{n=1}^{\infty}\frac{d(n)\e{n h/k}}{n^s}.
}
The Estermann function extends analytically to $\C\setminus\{1\}$ and satisfies a functional equation; these properties are useful in studying the asymptotics of the mean square of the Riemann zeta function multiplied by a Dirichlet polynomial (see~\cite{BCHB}), which are needed, for example, for theorems which give a lower bound for the portion of zeros of $\zeta(s)$ on the critical line. See also~\cite{C} and~\cite{Iw}. The sum
\est{
V\pr{\frac hk}:=\sum_{m=1}^{k-1}\pg{\frac {mh}k}\cot\pr{\frac{\pi  m}{k}}=-c_0\pr{\frac {\overline h}k},
}
known as the Vasyunin sum, arises in the study of the Riemann zeta function by virtue of the formula:
\es{\label{vasf}
\nu\pr{\frac hk}:=&\,\frac1{2\pi\sqrt{hk}}\int_{-\infty}^\infty\pmd{\zeta\pr{\frac12+it}}^2\pr{\frac hk}^{it}\frac{dt}{\frac14+t^2}\\
=&\frac{\log2\pi-\gamma}2\pr{\frac1h+\frac1k}+\frac{k-h}{2hk}\log\frac hk-\frac{\pi}{2hk}\pr{V\pr{\frac hk}+V\pr{\frac kh}}.
}
{
\begin{figure}[ht]
\begin{minipage}[b]{0.5\linewidth}
\centering
\includegraphics[height=108pt,width=190pt]{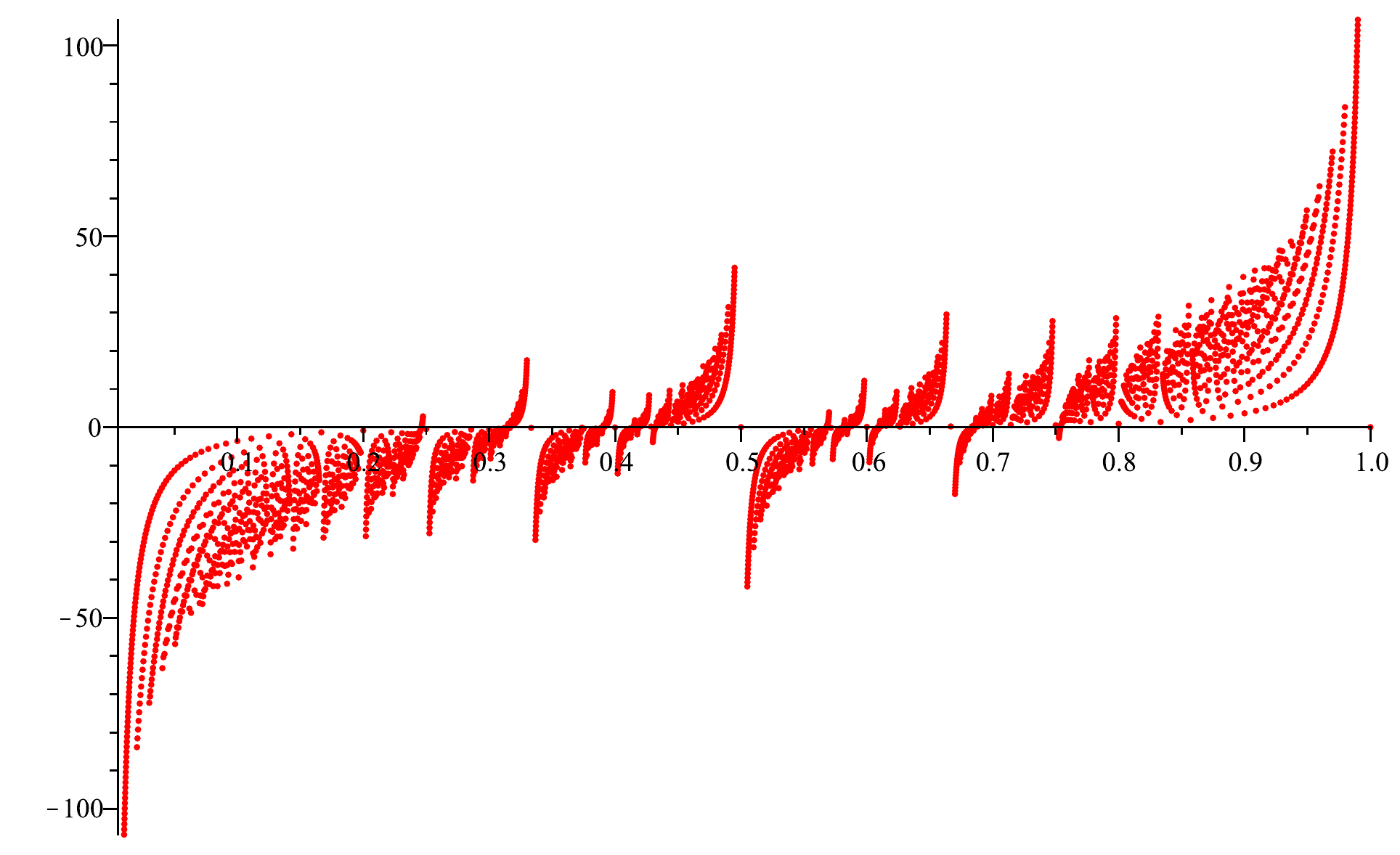}
\caption{Graph of $V\pr{\frac hk}$ for $1\leq h,k\leq 100$ and $(h,k)=1$.}
\label{fig:F3}
\end{minipage}
\hspace{0.5cm}
\begin{minipage}[b]{0.5\linewidth}
\centering
\includegraphics[width=190pt]{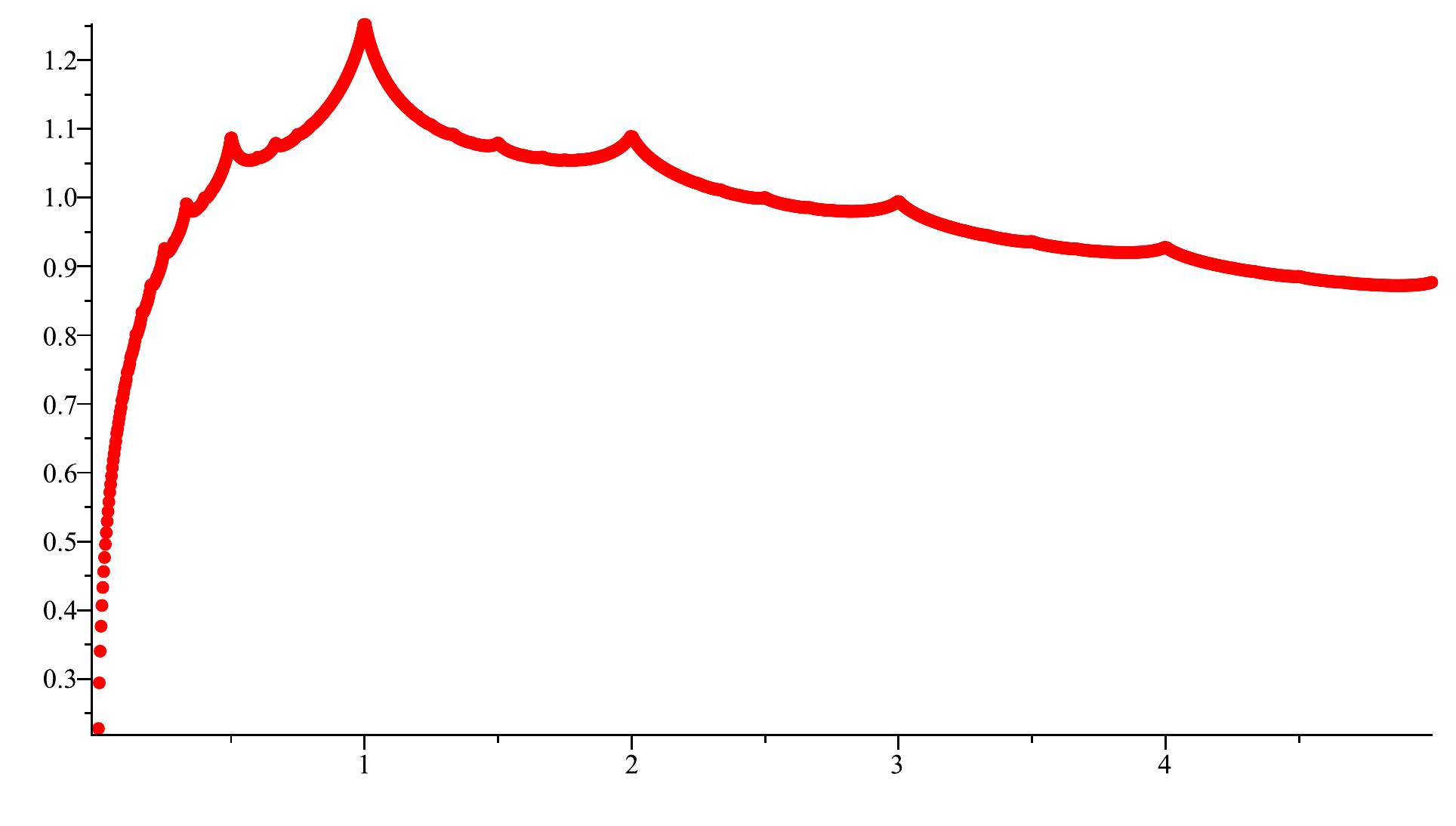}
\caption{Graph of $\sqrt{hk}\,\nu\pr{\frac hk}$ for $1\leq h\leq 5k$, $k=307$, $(h,k)=1$.}
\label{fig:F4}
\end{minipage}
\end{figure}
}
This formula is relevant to the Nyman-Beurling-Baez-Duarte-Vasyunin approach to the Riemann hypothesis which asserts that the Riemann hypothesis is true if and only if $\lim_{N\rightarrow\infty}d_N=0,$ where
\est{
d_N^2=\inf_{A_N}\frac1{2\pi}\int_{-\infty}^{\infty}\pmd{1-\zeta A_N\,\pr{\frac12+it}}^2\frac{dt}{\frac 14+t^2}
}
and the inf is over all the Dirichlet polynomials $A_N(s)=\sum_{n=1}^{N}\frac{a_n}{n^{s}}$ of length $N$; see~\cite{Bag} for a nice account of the Nyman-Beurling approach to the Riemann hypothesis with Baez-Duarte's significant contribution and see~\cite{BDBLS} and~\cite{LR} for information about the Vasyunin sums, as well as interesting numerical experiments about $d_N$ and the minimizing polynomials $A_N$. Thus $d_N^2$ is a quadratic expression in the unknown quantities $a_m$ in terms of the Vasyunin sums.

In~\cite{BC} we showed that $c_0\pr{\frac hk}$ satisfies the reciprocity formula
\es{\label{fc0}
c_0\pr{\frac hk}+\frac kh c_{0}\pr{\frac kh}-\frac1{\pi h}&=\frac i2\psi_{0}\pr{\frac hk}
}
(and in particular that $c_0\pr{\frac hk}$ can be computed to within a prescribed accuracy in a time that is polynomial in $\log k$).

{
\begin{figure}[h!]
  \centering
\includegraphics[scale=0.40]{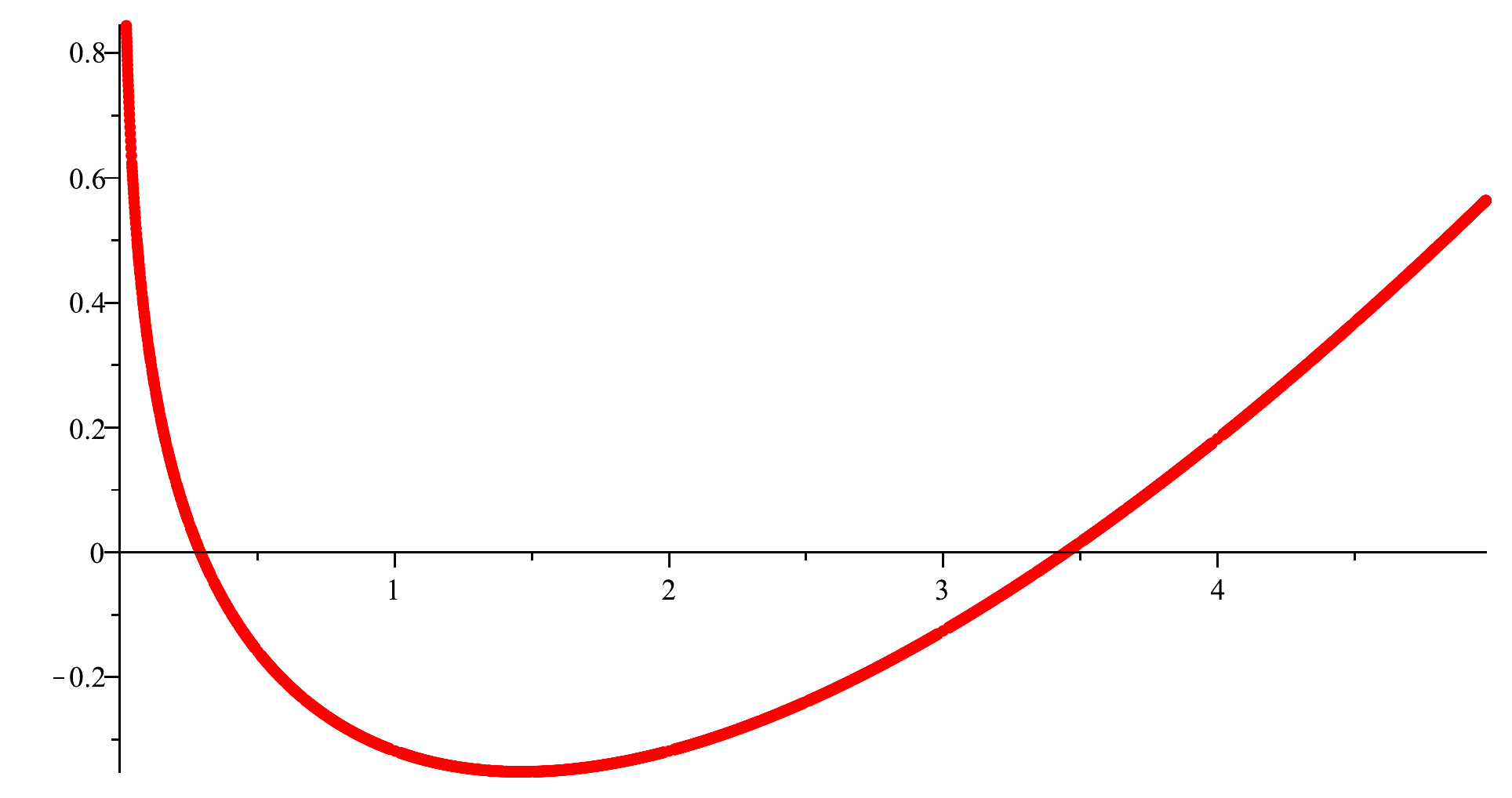}
  \caption{Graph of $c_0\pr{\frac hk}+\frac kh c_{0}\pr{\frac kh}-\frac1{\pi h}$ for $h\leq 5k$, $k\leq50$ and $(h,k)=1$.}
  \label{fig:F5}
\end{figure}
} 

This behavior is analogous to that of the Dedekind sum,
\est{
s\pr{\frac hk}=-\frac1{4k}\sum_{m=1}^{k-1}\cot\pr{\frac {\pi m}k}\cot\pr{\frac{\pi m h}k},
}
which satisfies the well known reciprocity formula
\es{\label{rs}
s\pr{\frac hk}+s\pr{\frac kh}-\frac{1}{12hk}=\frac1{12}\pr{\frac{h}{k}+\frac{k}{h}-3}.
}

In this paper we prove that these results can be generalized to the sums
\est{
c_a\pr{\frac hk}:=k^a\sum_{m=1}^{k-1}\cot\pr{\frac{\pi  mh}{k}}\zeta\pr{-a,\frac mk},
}
where $\zeta(s,x)$ is the Hurwitz zeta function (note that at $a=-1$ the poles of $\zeta\pr{-a,\frac mk}$ cancel). 

Notice that, for all $a$, $c_a\pr{\frac hk}$ is odd and periodic in $x=\frac hk$ with period 1 and, for non-negative integers $a$, it takes values in the cyclotomic field containing all roots of
unity.

At the non-negative integers, $a=n\geq0$, these cotangent sums can be expressed in terms of the Bernoulli polynomials,
\est{
c_{n}\pr{\frac hk}=-k^n\sum_{m=1}^{k-1}\cot\pr{\frac{\pi  mh}{k}}\frac{B_{n+1}\pr{\frac mk}}{n+1},
}
most interestingly in the case when $n$ is even, since $c_n\equiv0$ for positive odd $n$.

If $a=-n$ is a negative integer one can write $c_a$ as
\est{
c_{-n}\pr{\frac hk}=\frac{(-1)^{n}}{k^n(n-1)!}\sum_{m=1}^{k-1}\cot\pr{\frac{\pi  mh}{k}}\Psi\pr{n-1,\frac mk},
}
where $\Psi(m,z):=\frac{d^{m+1}}{dz^{m+1}}\log\Gamma(z)$ is the polygamma function.

By the reflection formula for the polygamma function,
\est{
\Psi(m,1-z)+(-1)^{m+1}\Psi(m,z)=(-1)^m\pi\frac{d^m}{dz^m}\cot(\pi z),
}
for a positive odd integer $n$ we can write $c_{-n}$ as
\est{
c_{-n}\pr{\frac hk}=-\frac{\pi}{2k^n(n-1)!}\sum_{m=1}^{k-1}\cot\pr{\frac{\pi  mh}{k}}\left.\frac{d^{n-1}}{dz^{n-1}}\cot\pr{\pi z}\right|_{z=\frac mk}
}
and, in particular,
\est{
c_{-1}\pr{\frac hk}=2\pi s\pr{\frac hk}.
}

Like the case $a=0$, these cotangent sums appear in the value at $s=0$,
\es{
D\pr{0,a,\frac hk}=-\frac12{\zeta(-a)} +\frac i2 c_{a}\pr{\frac hk},
}
of the function $D\pr{s,a,\frac hk}$, defined for $\Re(s)>1$ by
\est{
D\pr{s,a,\frac hk}:=\sum_{n=1}^{\infty}\frac{\sigma_a(n)\e{n h/k}}{n^s}.
}
Moreover, the cotangent sums $c_a$ appear also in a shifted version of the Vasyunin formula~\eqref{vasf} (see Theorem~\ref{vasp} at the end of the paper for a new analytic proof).

\begin{theo}\label{tca}
Let $h,k\geq1$, $(h,k)=1$. Then
\es{\label{fca}
c_a\pr{\frac hk}-\pr{\frac kh}^{1+a} c_{a}\pr{\frac {-k}h}+k^a\frac{a\zeta(1-a)}{\pi h}=-i\zeta(-a)\psi_a\pr{\frac hk}.
}
\end{theo}
(Note that, since $g_{-1}(z)$ is identically zero, for $a=-1$ the reciprocity formula reduces to~\eqref{rs}.) In particular, $c_a\pr{\frac hk}$ gives an example of an ``imperfect" quantum modular form of weight $1+a$.

New formulae can be obtained by differentiating~\eqref{fca}; for example, writing
\est{
c_{-1}^*\pr{\frac hk}:=\frac1{k}\sum_{m=1}^{k-1}\cot\pr{\frac{\pi  mh}{k}}\gamma_1\pr{\frac mk},
}
where $\gamma_1(x)$ is the first generalized Stieltjes constant defined by
\est{
\zeta(s,x)=\frac{1}{s-1}+\sum_{n=0}^\infty \frac{(-1)^n}{n!} \gamma_n(x) (s-1)^n,
}
taking the derivative at $-1$ of~\eqref{fca} multiplied by $k^{-a}$ we get the formula
\est{
c_{-1}^*\pr{\frac hk}&-c_{-1}^*\pr{\frac {-k}h}+\frac{\zeta'(2)+\frac{\pi^2}{6}}{\pi kh}+\pi\log k\pr{\frac {1}{6 }\frac kh-\frac12}=q\pr{\frac hk},
}
where
\est{
q(z):=-\frac{1}{\pi z}\zeta'(2)+\frac\pi2\pr{\log z+\gamma}+g'_{-1}\pr{z}
}
is holomorphic in $\C'$.

\section{The period function}\label{MS}

In this section we give a proof of Theorem~\ref{tb} and~\ref{tb2}.

\begin{proof}[Proof of Theorem~\ref{tb}]
Firstly, observe that the three term relation~\eqref{3tr} follows easily from the periodicity in $z$ of $E(a,z)$.

$\s_a(z)$ can be written as
\es{\label{E(z)}
\s_a(z)&=\sum_{n=1}^{\infty}\sigma_a(n)\frac1{2\pi i}\int_{\pr{2+\max\pr{0,\Re(a)}}}\Gamma(s)(-2\pi inz)^{-s}\,ds\\
&=\frac1{2\pi i}\int_{\pr{2+\max\pr{0,\Re(a)}}}\zeta(s)\zeta(s-a)\Gamma(s)e^{\frac{\pi is}2}(2\pi z)^{-s}\,ds\\
&=\frac1{2\pi i}\int_{\pr{-\frac12-2M}}\zeta(s)\zeta(s-a)\Gamma(s)e^{\frac{\pi is}2}(2\pi z)^{-s}\,ds+r_{a,M}(z),
}
where $M$ is any integer greater or equal to $-\frac12\min\pr{0,{\Re(a)}}$ and
\est{
r_{a,M}(z):=&-\frac12\zeta(-a)+i\frac{\zeta(1-a)}{2\pi z}+i\frac{\zeta(1+a)\Gamma(1+a)e^{\frac{\pi ia}2}}{(2\pi z)^{1+a}}+\\
&+\sum_{1\leq n\leq M}i(-1)^{n}\frac{B_{2n}}{(2n)!}\zeta(1-2n-a)(2\pi z)^{2n-1}
}
is the sum of the residues encountered moving the integral (and has to be interpreted in the limit sense if some of the terms have a pole). Now, consider
\est{
\frac{1}{z^{1+a}}\s_a\pr{-\frac1z}&=\frac{1}{z^{1+a}}\frac1{2\pi i}\int_{\pr{2+\max\pr{0,\Re(a)}}}\zeta(s)\zeta(s-a)\Gamma(s)e^{\frac{\pi is}2}\pr{2\pi \frac{-1}z}^{-s}\,ds\\
&=\frac1{2\pi i}\int_{\pr{2+\max\pr{0,\Re(a)}}}\zeta(s)\zeta(s-a)\Gamma(s)e^{-\frac{\pi is}2}\pr{2\pi}^{-s}z^{s-1-a}\,ds,\\
}
since in this context $0<\arg z<\pi$ and $0<\arg \frac{-1}z<\pi$, so the identity $\arg\frac{-1}z=\pi-\arg z$ holds. Applying the functional equation to both $\zeta(s)$ and $\zeta(s-a)$ we get, after the change of variable $s\rightarrow1-s+a$,
\es{\label{E(-1/z)}
\frac{1}{z^{1+a}}\s_a\pr{-\frac1z}&=-\frac1{2\pi}\int_{\pr{- 1+\min\pr{0,\Re(a)}}}\zeta(s-a)\zeta(s)\Gamma(s)\frac{e^{\frac{\pi i\pr{s-a}}2}\cos\frac {\pi s}2}{\sin\frac {\pi \pr{s-a}}2}(2\pi z)^{-s}\,ds\\
&=-\frac1{2\pi}\int_{\pr{- \frac12-M}}\zeta(s-a)\zeta(s)\Gamma(s)\frac{e^{\frac{\pi i\pr{s-a}}2}\cos\frac {\pi s}2}{\sin\frac {\pi \pr{s-a}}2}(2\pi z)^{-s}\,ds,
}
since the integrand doesn't have any pole on the left of $- 1+\min\pr{0,\Re(a)}$. The theorem then follows summing~\eqref{E(z)} and~\eqref{E(-1/z)} and using the identity
\est{
e^{\frac{\pi is}2}+i\frac{e^{\frac{\pi i\pr{s-a}}2}\cos\frac {\pi s}2}{\sin\frac {\pi \pr{s-a}}2}=i\frac{\cos\frac{\pi a}2}{\sin\frac {\pi \pr{s-a}}2}.
}
\end{proof}

We remark that for $a=2 k+1$, $k\geq1$, Theorem~\ref{tb} reduces to
\est{
E_{2k}(z) -\frac{1}{z^{2k}} E_{2k}\pr{-\frac1z}=0,
}
while, for $a=1$, the theorem reduces to the well known identity
\est{
E_2\pr{z}-\frac1{z^2}E_2\pr{\frac{-1}z}=-\frac{12}{2\pi i z}.
}

To prove Theorem~\ref{tb2} we need the following lemma.

\begin{lemma}\label{asympt}
For fixed complex numbers $A$ and $\alpha$ we have, as $n\rightarrow\infty$
\est{
J_n:=\int_0^\infty u^{n+\alpha} e^{-A\sqrt u} e^{-u}\frac {du} u=\sqrt{2\pi}e^{\frac {A^2}8}e^{-A\sqrt n}e^{-n}n^{n+\alpha-\frac12}\pr{1-\frac C{\sqrt n}+O\pr{\frac 1n}},
}
where
\est{
C=\frac{4\alpha-1}{8}A+\frac{A^3}{96}.
}
\end{lemma}

\begin{proof}
After the change of variable $u=nx^2$ we have
\est{
J_n=&\,2n^{n+\alpha}\int_0^\infty x^{2\alpha-1} e^{-A\sqrt{n}x-n\pr{x^2-2\log x}} {dx} \\
=&\,2n^{n+\alpha}e^{-A\sqrt{n}}\int_{-1}^\infty (x+1)^{2\alpha-1} e^{-A\sqrt{n}x-n\pr{(x+1)^2-2\log (x+1)}} {dx}\\
=&\,2n^{n+\alpha}e^{-A\sqrt{n}}e^{-n}\pr{1+O\pr{e^{-\frac{n\delta^2}2}}}\times\\
&\times\int_{-\delta}^\delta (x+1)^{2\alpha-1} e^{-A\sqrt{n}x-2nx^2}\pr{1+\frac{2nx^3}{3}+O\pr{nx^4} }{dx},
}
for any small $\delta>0$. We can then approximate the binomial and extend the integral to $\R$ at a negligible cost, getting
\est{
J_n&=2n^{n+\alpha}e^{-A\sqrt{n}}e^{-n}\int_{-\infty}^\infty \pr{1+\pr{2\alpha-1}x+\frac{2nx^3}{3}+O\pr{x^2+nx^4}}\times\mbox{}\\
&\hspace{120pt}\times e^{-A\sqrt{n}x-2nx^2}{dx}.\\
}
Evaluating the integrals, the lemma follows.
\end{proof}

\begin{proof}[Proof of Theorem~\ref{tb2}]
The three term relation~\eqref{3tr} implies that
\est{
g_a(z+1)=&\,\frac1{(z+1)^{1+a}}\cot\frac{\pi a}2\zeta(-a)-\frac1{\pi z(z+1)^{a}}\zeta(1-a)+\mbox{}\\
&+\frac 1{\pi z(z+1)}\zeta(1-a)+g_a(z)-\frac1{(z+1)^{1+a}}g_a\pr{\frac z{z+1}}.\\
}
Now, from the definition~\eqref{g_a} of $g_a(z)$, it follows that
\est{
g_a(z)=2\sum_{1\leq n\leq M}(-1)^{n}\frac{B_{2n}}{(2n)!}\zeta(1-2n-a)(2\pi z)^{2n-1}+O\pr{|z|^{2M+\frac12}},
}
for any $M\geq 1$. Thus
\est{
&g_a(z)-\frac{g_a\pr{\frac z{z+1}}}{(z+1)^{1+a}}=\\
&=2\sum_{1\leq n\leq M}(-1)^{n}\frac{B_{2n}}{(2n)!}\zeta(1-2n-a)(2\pi z)^{2n-1}\pr{1-\frac1{(z+1)^{2n+a}}}+O\pr{|z|^{2M+\frac12}}\\
&=-2\sum_{m=2}^{2M} \Bigg(\sum_{\substack{2n-1+k=m,\\ n, k\geq1}}(-1)^{n+m}B_{2n}\zeta(1-2n-a)\frac{\Gamma(2n+a+k)}{\Gamma(2n+a)k!(2n)!}(2\pi )^{2n-1}\Bigg)z^m\\
&\quad+O\pr{|z|^{2M+\frac12}}.
}
Therefore,
\est{
g_a(z+1)=\sum_{m=0}^{2M}b_m z^m+O(|z|^{2M+\frac12}),
}
where
\est{
b_m&:=-2\sum_{\substack{2n-1+k=m,\\ n, k\geq1}}(-1)^{n+k}B_{2n}\zeta(1-2n-a)\frac{\Gamma(2n+a+k)}{\Gamma(2n+a)k!(2n)!}(2\pi )^{2n-1}+\\
&\quad+(-1)^m\cot\frac{\pi a}2\zeta(-a)\frac{\Gamma(1+a+m)}{\Gamma(1+a)m!}+\mbox{}\\
&\quad+(-1)^m\pr{\frac{\Gamma(1+a+m)}{\Gamma(a)(m+1)!}-1}\frac{\zeta(1-a)}{\pi},
}
and, since $g_a(z)$ is holomorphic at 1, $b_m$ must coincide with the $m$-th coefficient of the Taylor series of $g_a(z)$ at 1.

Now, let's prove the asymptotic~\eqref{aseq}. Fix any $M\geq -\frac12\min\pr{0,{\Re(a)}}$ and assume $m\geq2M+1$ and $\Re(\tau)>0$. By the functional equation for $\zeta$ and basic properties of $\Gamma(s)$, we have
\est{
&\frac{(2\pi)^a\tau^m}{\cos\frac{\pi a}2}g^{\pr{m}}_a(\tau)=\\
&\quad=\frac{(-1)^m}{\pi i}\int_{\pr{-\frac12-2M}}\Gamma(s)\frac{\zeta(s)\zeta(s-a)}{\sin\frac {\pi \pr{s-a}}2}s\pr{s+1}\cdots\pr{s+m-1}(2\pi)^{-s+a}\tau^{-s}\,ds\\
&\quad=\frac{(-1)^m}{\pi i}\int_{\pr{-\frac12-2M}}\frac{\zeta(s)\zeta(s-a)}{\sin\frac {\pi \pr{s-a}}2}\Gamma(s+m)(2\pi)^{-s+a}\tau^{-s}\,ds\\
&\quad=\frac{(-1)^m}{\pi^3 i}\int_{\pr{-\frac12-2M}}\zeta(1-s)\zeta(1-s+a)\times\mbox{}\\
&\hspace{90pt}\times\Gamma(1-s)\Gamma(1-s+a)\Gamma(s+m)\sin\frac{\pi s}2 \pr{\frac{2\pi}\tau}^{s}\,ds.\\
}
We can see immediately that $g^{\pr{m}}_a(\tau)\ll_am^{-B}|\tau|^{-m}m!$ for any fixed $B>0$, just by moving the path of integration to the line $\Re(s)=-B$ and using trivial estimates for $\Gamma$. To get a formula which is asymptotic as $m\rightarrow\infty$ we expand $\zeta(1-s)\zeta(1-s+a)$ into a Dirichlet series and integrate term-by-term; the main term arises from the first term of the sum. We have
\est{
g^{\pr{m}}_a(\tau)=2\frac{(-\tau)^{-m}\cos\frac{\pi a}2}{\pi^2(2\pi)^a}\sum_{\ell=1}^{\infty}\frac{\sigma_{-a}(\ell)}{\ell} I_{m,a}\pr{\frac {\ell}{\tau}},
}
where
\est{
I_{m,a}(x):=\frac1{2\pi i}\int_{\pr{-\frac12-2M}}\Gamma(1-s)\Gamma(1-s+a)\Gamma(s+m)\sin\frac{\pi s}2 (2\pi x)^{s}\,ds.
}
We re-express this integral as a convolution integral. Recall that for $|\arg x|<\pi$  we have
\est{
\frac1{2\pi i}\int_{\pr{\frac32+2M}}\Gamma(s)\Gamma(s+a) u^{-s}\,ds=2u^{\frac a2}K_{a}\pr{2\sqrt u},
}
where $K_a$ denotes the K-Bessel function of order $a$. Also,
\est{
\frac1{2\pi i}\int_{\pr{-\frac12-2M}}\Gamma(s+m) u^{-s}\,ds=u^{m}e^{-u}.
}
Thus,
\est{
I_{m,a}(x)=I_{m,a}^{+}(x)+I_{m,a}^{-}(x),
}
where
\est{
I_{m,a}^{\pm}(x)=(2\pi x)^{1+\frac a2}e^{\frac{\pm \pi i a}4}\int_0^\infty u^{m+\frac a2}K_{a}\pr{2e^{\pm \frac{\pi i}4}\sqrt {2\pi x u}}e^{-u}\,du.
}
Now, for $|\arg z|<\frac32\pi$
\est{
K_a(z)=\sqrt{\frac{\pi}{2z}}e^{-z}\pr{1+\frac{4a^2-1}{8z}+O_a\pr{\frac1{|z|^2}}},\\
}
as $z\rightarrow\infty$, and
\est{
K_{-a}(z)=K_a(z)\sim
\begin{cases}
2^{a-1}\Gamma(a)z^{-a}, &\text{if $\Re(a)\geq 0$, $a\neq0$,}\\
-\log\frac x2 -\gamma,&\text{if $a=0$,}\\
\end{cases}
}
as $z\rightarrow0$. Therefore, by Lemma~\ref{asympt},
\est{
I_{m,a}^\pm (x)=&\,(2\pi x)^{1+\frac a2}\frac{\pi^\frac14 e^{\frac{\pm \pi i \pr{a-\frac12}}4}}{2^\frac54  {x}^\frac14}\int_0^\infty u^{m+\frac a2-\frac14}e^{-u-2(1\pm i)\sqrt {\pi x u}}	\times\\
&\hspace{70pt}\times \pr{1+\frac{4a^2-1}{2^\frac92\pi^\frac12 e^{\pm\frac{\pi i}4}\sqrt {x u}}+O_a\pr{\frac1{u}}}\,du\\
\sim&\, 2^{\frac14+\frac a2}\pi^{\frac74+\frac a2} e^{\frac{\pm \pi i \pr{a-\frac12}}4}x^{\frac34+\frac a2}e^{ {\pm i\pi x}}e^{-2(1\pm i)\sqrt {\pi xn}} e^{-m}m^{m+\frac14 +\frac a2}\times\mbox{}\\
&\hspace{70pt}\times\pr{1+\frac {\xi^\pm}{\sqrt m}+O\pr{\frac 1m}},\\
}
where
\est{
\xi^\pm=-\frac{(1\pm i)\sqrt{\pi x}(1+a)}{2}+\frac{(1\mp i)(\pi x)^\frac 32}{6}+\frac{\pr{4a^2-1}(1\mp i)}{32 \pi^\frac12 \sqrt x},
}
and~\eqref{aseq} follows.
\end{proof}

\section{An extension of Voronoi's formula}\label{VF}

Formula~\eqref{fbF} can be proved with the same techniques used to prove Theorems~\ref{tb} and~\ref{tb2}. In this section we give an application of this formula and we discuss a similar formula for convolutions of the exponential function. We conclude the section showing how these results can be used to prove Voronoi's formula.

Applying formula~\eqref{fbF} to $F(s)=\frac{\Gamma\pr{\frac{s}2}}{2\Gamma(s)}$ we get, for $\frac{\pi }4<\arg (z)<\frac 34\pi$,
\es{\label{exap}
\sum_{n=1}^\infty d(n)e^{(2\pi n z)^2}=\frac1z\sum_{n=1}^\infty  d(n)T(4\pi n z)+R(z)+k(z),
}
where, for $\frac{\pi }4<\arg (z)<\frac 34\pi$,
\est{
T(z):=&\frac1{\sqrt{\pi} i}\int_{\pr 2}\frac{\Gamma(s)}{\Gamma\pr{1-\frac {s}2}}(-iz)^{-s}\,ds=\sum_{n=0}^{\infty}\frac{(iz)^n}{n!\Gamma\pr{1+\frac n2}}\\
}
and
\est{
R(z):=&=\frac14+\frac{2\log \pr{-4\pi iz}-3\gamma}{8\sqrt \pi i z},\\
k(z):=&\frac1{4\pi^2 }\int_{\pr{-\frac12}}\Gamma\pr{\frac s2}\Gamma(1-s)\zeta(s)\zeta(1-s)z^{-s}\,ds.
}
Notice that we have $T(z)\ll |z|^{-B}$ for all fixed $B>0$; moreover, $k(z)$ is holomorphic in $|\arg (z)|<\frac 34\pi$ and, if $|\arg\pr{\tau}|<\frac{\pi}{4}$,
\est{
c_\tau(m):= \frac{k^{(m)}(\tau)}{m!}\ll |\tau|^{-m} m^{-B}
}
for all $B>0$. In particular, if we set $z=i\delta$ with $0<\delta\leq1$, taking the real part of~\eqref{exap} we get
\es{
\sum_{n=1}^\infty d(n)e^{-(2\pi n \delta)^2}=\frac14+\frac{-2\log(4\pi \delta)-3\gamma}{4\sqrt\pi \delta}+\Re\sum_{m=0}^\infty c_m\pr{\frac{ \sqrt 3}2+i\pr{\frac12-\delta}}^m
}
with 
\est{
c_m:=c_{\frac{\sqrt3+i}2}(m)\ll m^{-B}.
}
for all $B>0$. 

We now state a similar formula for convolutions of the exponential function and a function that is compactly supported on $\R_{>0}$. 

Let $g(x)$ be a compactly supported function on $\R_{>0}$ and let
\est{
W_{+}\pr{z}:=&\,\int_0^{\infty}f\pr{\frac1x}\e{zx}\,\frac{dx}{x} \\
W_{-}\pr{z}:=&\, \int_0^{\infty}f\pr{x}\e{zx}\,dx.\\
}
If we denote the Mellin transform of $f(x)$ with $F(s)$, then it follows that $F(s)$ is entire and that $W_{+}(x)$ and $W_{-}(x)$ can be written as in~\eqref{Wpm}. In particular, since
\est{
F(0)&=\int_{0}^\infty f(x)\,\frac{dx}x,\\
F(1)&=\int_{0}^\infty f(x)\,{dx},\\
F'(1)&=\int_{0}^\infty f(x)\log x\,{dx},\\
}
formula~\eqref{fbF} can be written as 
\es{
\sum_{n=1}^\infty d(n)&W_{+}\pr{nz}-\frac1z\sum_{n=1}^\infty d(n)W_{-}\pr{-\frac nz}=\\
=&\,\int_{0}^\infty f(x)\pr{\frac1{4x}-\frac{{1 }}{4  z} -\frac{\gamma-\log(2\pi z/x)}{2\pi i z}}\,dx+k(z)+\mbox{}\\
&+\int_{0}^\infty f(x)\int_{\pr{-\frac12}}\frac{\zeta(s)\zeta(1-s)}{\sin\pi s}\pr{\frac zx}^{-s}\,ds\,\frac{dx}{2\pi x},\\
}
for $\Im(z)>0$.

\begin{proof}[Proof of Voronoi's formula]
Let $f:\R_{\geq0}\rightarrow\R$ be a smooth function that decays faster than any power of $x$ and let 
\est{
\tilde f(x):=2\int_{0}^\infty f(y)\cos\pr{2\pi xy}\,dy
}
be the cosine transform of $f(x)$. Then, $\tilde f(x)$ is smooth and, by partial integration, $\tilde f^{(m)}(x)\ll\frac1{x^{2+m}}$ for all $m\geq0$. For  $0<\Re(s)<2$, we can define the Mellin transform of $\tilde f$, 
\est{
F(s):=\int_0^\infty \tilde f(x)x^{s-1}\,dx.
}
By partial integration we see that $F(s)$ extends to a meromorphic function on $\Re(s)<2$ with simple poles at most at the non-positive integers. Also, $F(s)$ decays rapidly on vertical strips. Moreover, by Parseval's formula, for $0<\Re(s)<1$ we have
\est{
F(s)&=\frac 2s\int_0^{\infty}f(y)(2\pi y)^{-s}\Gamma(s+1)\cos\frac{\pi s}2\,dy\\
&=\frac{2}s\int_0^{\infty}f(y)\,dy-2\int_0^{\infty}f(y)\pr{\log (2\pi y)+\gamma}\,dy+O(|s|)\\
&=\frac{F_{-1}}s+F_0+O(|s|),
}
say. For $\Im (z)\geq0$ we can define
\es{
W_+(z):=&\,\frac1{2\pi i}\int_{\pr{\frac32}}F(s)\Gamma(s)\pr{-2\pi iz}^{-s}\, ds=\int_0^{\infty}\tilde f\pr{\frac1x}\e{zx}\,\frac{dx}{x},\\
W_-(z):=&\,\frac1{2\pi i}\int_{\pr{\frac32}}F(1-s)\Gamma(s)\pr{-2\pi iz}^{-s}\, ds\\
=&\,\int_0^{\infty}\pr{\tilde f\pr{x}-\tn{Res}_{s=0}F(s)}\e{zx}\,dx,
}
with the second representation of $W_-(z)$ defined only on $\Im(z)>0$. Since $F(s)$ is rapidly decaying at infinity,~\eqref{fbF} holds for $\Im(z)\geq0$ and so we can apply that formula for $z=1$ and take the real part. By the definition of $\tilde f$,  we have
\est{
\Re\pr{W_+(n)}&=2\int_0^{\infty} f(y)\int_{0}^\infty \cos\pr{ \frac {2\pi y}x} \cos\pr{nx}\,\frac{dx}{x}\,dy\\
&=\int_0^\infty f(y)\pr{2K_0\pr{4\pi\sqrt{ny}}-\pi Y_0\pr{4\pi\sqrt{ny}}}\,dy\\
}
and
\est{
\Re\pr{W_-(-n)}&=\lim_{\substack{z\rightarrow1,\\ \Im(z)>0}}\Re\pr{W_-(-nz)}\\
&=\lim_{\substack{z\rightarrow1,\\ \Im(z)>0}}\Re\int_0^{\infty}\!\!\tilde f\pr{x}\e{-nzx}dx-\lim_{\substack{z\rightarrow1,\\ \Im(z)>0}}\Re\frac{\tn{Res}_{s=0}F(s)}{-2\pi i nz}=\frac12 f(n),
}
since $\tn{Res}_{s=0}F(s)$ is real. Moreover,  $\frac1{2\pi }\int_{\pr{-\frac12}}F(s)\frac{\zeta(s)\zeta(1-s)}{\sin\pi s} z^{-s}\,ds$ is purely imaginary on the real line, so we just need to compute
\est{
\Re\bigg(\mathop{\tn{Res}}_{s=0,1}F(s)\Gamma(s)\zeta(s)^2&(-2\pi i)^{-s}\bigg)=\\
=&\Re\Bigg(\frac{F(1)\pr{\gamma-\log(-2\pi i )}+F'(1)}{-2\pi i }\\
&\qquad+\frac{-F_{-1}\pr{\log (-2\pi i )+\gamma-2\log2\pi}+F_0}{4}\Bigg)\\
=&-\frac{  f(0)}{8}-\frac12\int_0^{\infty}f(y)\pr{\log y+2\gamma}\,dy,
}
since $F(1)=\frac{f(0)}2$ and $F'(1)$ is real. This completes the proof of the theorem.
\end{proof}

\section{An exact formula for the second moment of $\zeta(s)$}
In this section we prove the exact formula for the second moment of the Riemann zeta function.
\begin{proof}[Proof of Theorem~\ref{tr}]
Firstly, observe that
\est{
L_2(\delta)&=-ie^{-\frac{i\delta}2}\int_{\frac12}^{\frac12+i\infty}\zeta(s)\zeta(1-s)e^{i\delta s}\,ds.
}
The functional equation for $\zeta(s)$,
\est{
\zeta(1-s)=\chi(1-s)\zeta(s),
}
where
\est{
\chi(1-s)=(2\pi)^{-s}\Gamma(s)\left(e^{\frac{\pi is}2}+e^{-\frac{\pi is}2}\right),
}
allows us to split $L_2(\delta)$ as
\begin{equation*}
\begin{split}
L_2(\delta)&=-ie^{-\frac{i\delta}2}\int_{\frac12}^{\frac12+i\infty}\chi(1-s)\zeta(s)^2e^{i\delta s}\,ds\\
&=-ie^{-\frac{i\delta}2}\left(L^+(\delta)+L^-(\delta)\right),
\end{split}
\end{equation*}
where
\begin{equation*}
\begin{split}
L^\pm(\delta)=\int_\frac12^{\frac12+i\infty}(2\pi)^{-s}\Gamma(s)e^{\pm\frac{\pi i s}{2}}\zeta(s)^2e^{i\delta s}\,ds.
\end{split}
\end{equation*}
By Stirling's formula $L^+(\delta)$ is analytic for $\Re(\delta)>-\pi$. Moreover, by contour integration,
\begin{equation*}
\begin{split}
L^-(\delta)&=\int_{(2)}(2\pi)^{-s}\Gamma(s)e^{-\frac{\pi i s}{2}}\zeta(s)^2e^{i\delta s}\,ds-G(\delta)\\
&=J(\delta)-G(\delta),
\end{split}
\end{equation*}
say, where
\begin{equation*}
\begin{split}
G(\delta):=&\int_{\frac12-i\infty}^\frac12(2\pi)^{-s}\Gamma(s)e^{-\frac{\pi i s}{2}}\zeta(s)^2e^{i\delta s}\,ds+\mbox{}\\
&+2\pi i\,\textnormal{Res}_{s=1}\left((2\pi)^{-s}\Gamma(s)e^{-\frac{\pi i s}2}\zeta(s)^2e^{i\delta s}\right)
\end{split}
\end{equation*}
is analytic  for $\Re(\delta)<\pi$. Now, expanding $\zeta(s)^2$ into its Dirichlet series, for $\Re(\delta)>0$ we have
\begin{equation}\label{aaa}
\begin{split}
J(\delta)&=\sum_{n= 1}^\infty d(n)\int_{2-i\infty}^{2+i\infty}\Gamma(s)(2\pi i ne^{-i\delta})^{-s}\,ds\\
&=2\pi i\s_0\left(-e^{-i\delta}\right)=2\pi i\s_0\left(1-e^{-i\delta}\right).
\end{split}
\end{equation}
By Theorem~\ref{tb}, we can write this as
\est{
J(\delta)&=\frac {\log2\pi\delta-\gamma}{1-e^{-i\delta}}-\pi g_0\pr{1-e^{-i\delta}}+\frac{2\pi i}{1-e^{-i\delta}}\s_0\left(\frac {-1}{1-e^{-i\delta}}\right)+ie^{i\delta}\omega(\delta),
}
where
\est{
\omega(\delta)=-\frac{\log\pr{\frac{1-e^{-i\delta}}{\delta}}-\frac{\pi i}2}{2\sin\pr{\frac\delta 2}}
}
is holomorphic in $\pmd{\Re(\delta)}<\pi$. Summing up, we have
\es{\label{qf}
L_2(\delta)=&\frac{\gamma-\log2\pi\delta}{2\sin\frac\delta2}+\frac{\pi i}{\sin\frac\delta2}\s_0\left(\frac {-1}{1-e^{-i\delta}}\right)+i\pi e^{-\frac {i\delta}2}g_0\pr{1-e^{-i\delta}}\\
&+\omega(\delta)-ie^{-\frac{i\delta}2}\pr{L^+(\delta)-G(\delta)}.
}
The theorem then follows after writing
\est{
h(\delta):=i\pi e^{-i\frac{\delta}2}g_0(1-e^{-i\delta})
}
and applying Theorem~\ref{tb} and~\ref{tb2}.
\end{proof}

\section{Cotangent sums}
We start by recalling the basic properties of $D\pr{s,a,\frac{h}k}$.
\begin{lemma}\label{D}
For (h,k)=1, $k>0$ and $a\in\C$,
\est{
D\pr{s,a,\frac hk}-k^{1+a-2s}\zeta(s-a)\zeta(s)
}
is an entire function of $s$. Moreover, $D\pr{s,a,\frac hk}$ satisfies a functional equation,
\es{\label{fe}
D\pr{s,a,\frac hk}=&-\frac{2}{k}\pr{\frac{k}{2\pi}}^{2-2s+a}\Gamma\pr{1-s+a}\Gamma\pr{1-s}\times\\
&\times\Big(\cos\pr{\frac\pi2\pr{2s-a}}D\pr{1-s,-a,-\frac {\overline h}k}+\\
&\qquad-\cos\frac{\pi a}2\, D\pr{1-s,-a,\frac {\overline h}k}\Big),
}
and
\est{
D\pr{0,a,\frac hk}=\frac i2 c_{a}\pr{\frac hk}-\frac12\zeta\pr{-a}.
}
\end{lemma}

\begin{proof}
The analytic continuation and the functional equation for $D\pr{s,a,\frac{h}k}$ can be proved easily using the analogous properties for the Hurwitz zeta function and the observation that
\est{
D\pr{s,a,\frac hk}=&\frac1{k^{2s-a}}\sum_{m,n=1}^{k}\e{\frac{mn h}k}\zeta\pr{s-a,\frac mk}\zeta\pr{s,\frac nk}.
}
Moreover, applying this equality at 0, we see that
\est{
D\pr{0,a,\frac hk}&=-k^a\sum_{m,n=1}^{k-1}\e{\frac{mn h}k}\zeta\pr{-a,\frac mk}B_1\pr{\frac nk}-\frac{\zeta\pr{-a}}2\\
&=\frac i2 c_{a}\pr{\frac hk}-\frac{\zeta\pr{-a}}2,
}
where we used
\est{
\sum_{n=1}^{k-1}B_1\pr{\frac nk}\pr{\e{\frac{mh}k}}^n=-\frac12\frac{1+\e{\frac{mh}k}}{1-\e{\frac{mh}k}}=-\frac i2\cot\pr{\frac{\pi  mh}{k}},
}
that can be easily obtained from the equality
\est{
B_1(x)=\left.\frac{\tn{d}}{\tn{d}t}\pr{\frac{te^{xt}}{e^t-1}}\right|_{t=0}.
}
\end{proof}

\begin{proof}[Proof of Theorem~\ref{tca}]
Firstly, observe that we can assume $0\neq|a|<1$, since the result extends to all $a$ by analytic continuation. Now, taking $z=\frac{h}k\pr{1+i\delta}$, with $\delta>0$, we have
\est{
\s_a(z)&=\sum_{n\geq1}\sigma_a(n)\e{n\frac hk}e^{-2\pi n\frac hk\delta}\\
&=\frac1{2\pi i}\int_{\pr{2}}\Gamma(s)D\pr{s,a,\frac hk}\pr{2\pi \frac hk\delta}^{-s}\,ds.\\
}
Therefore, moving the integral to $\sigma=-\frac12$,
\est{
\s_a(z)=\frac{k^a}{2\pi h\delta} \zeta(1-a)+\frac{1}{(2\pi h\delta)^{1+a}}\zeta(1+a)\Gamma(1+a)+D\pr{0,a,\frac hk}+O\pr{\delta^\frac12}.
}
Similarly,
\est{
\frac1{z^{1+a}}\s_a\pr{\frac{-1}z}=&\,\frac1{z^{1+a}}\sum_{n\geq1}\sigma_a(n)\e{-n\frac{k}h}e^{-2\pi n \frac kh\frac{\delta}{1+i\delta}}\\
=&\,\frac{k^a}{2\pi \delta h}\zeta(1-a)+\frac{1}{(2\pi\delta h)^{1+a}}\zeta(1+a)\Gamma(1+a)\\
&-ia\frac{k^a}{2\pi h}\zeta(1-a)+\pr{\frac k{h(1+i\delta)}}^{1+a}D\pr{0,a,-\frac kh}+O\pr{\delta^\frac12}.
}
In particular, as $\delta$ goes to 0, we have
\est{
\s_a(z)-\frac1{z^{1+a}}\s_a\pr{\frac{-1}z}\longrightarrow& D\pr{0,a,\frac hk} -\pr{\frac k{h}}^{1+a}D\pr{0,a,-\frac kh}+\mbox{}\\
&+ia\frac{k^a}{2\pi h}\zeta(1-a).
}
Applying Theorem~\ref{tb}, it follows that
\est{
D\pr{0,a,\frac hk} -\pr{\frac k{h}}^{1+a}D\pr{0,a,-\frac kh}+&ia\frac{k^a}{2\pi h}\zeta(1-a)=\\
&=\frac{\zeta(-a)}2\pr{\pr{\frac kh}^{1+a}-1+\psi_a\pr{\frac hk}},
}
which is equivalent to~\eqref{fca}.
\end{proof}

We conclude the paper by giving a new proof of Vasyunin's formula (with a shift).

\begin{theo}\label{vasp}
Let $(h,k)=1$, $h,k\geq1$. Let $\pmd{\Re(a)}<1$. Then
\est{
&\frac{1+a}{2\pi }\int_{-\infty}^\infty \zeta\pr{\frac 12+\frac a2+it}\zeta\pr{\frac 12+\frac a2-it}\pr{\frac{h}k}^{-it}\frac{dt}{(\frac 12+\frac a2+it)(\frac12+\frac a2-it)}=\\
&=\,-\frac{\zeta\pr{1+a}}{ 2}\pr{\pr{\frac kh}^{\frac12+\frac a2}+\pr{\frac hk}^{\frac12+\frac a2}}+\frac{\zeta\pr{a}}{ a}\pr{\pr{\frac kh}^{\frac12-\frac a2}+\pr{\frac hk}^{\frac12-\frac a2}}+\\
&\quad -\pr{\frac{1}{hk}}^{\frac 12+\frac a2}(2\pi)^{a}\Gamma(-a)\sin\frac{\pi a}2\pr{c_{a}\pr{\frac {\overline h}k}+c_{a}\pr{\frac {\overline k}h}}.\\
}
\end{theo}

\begin{proof}
We need to evaluate
\est{
\frac{1+a}{2\pi (hk)^{\frac12+\frac a2}}&\int_{-\infty}^\infty \zeta\pr{\frac 12+\frac a2+it}\zeta\pr{\frac 12+\frac a2-it}\pr{\frac{h}k}^{it}\frac{dt}{(\frac 12+\frac a2+it)(\frac12+\frac a2-it)}=\\
&=\frac{1+a}{2\pi i}\int_{\pr{\frac 12-\frac {\Re(a)}2}}\frac{\zeta\pr{s+a}\zeta\pr{1-s}}{ h^{s+a}k^{1-s}}\frac{ds}{(s+a)(1-s)}.
}
We rewrite this as
\est{
\frac{1+a}{2\pi i}&\int_{\pr{\frac 12-\frac {\Re(a)}2}}\frac{\zeta\pr{s+a}\zeta\pr{1-s}}{ h^{s+a}k^{1-s}}\frac{ds}{(s+a)(1-s)}=\\
=&\,\frac1{2\pi i}\int_{\pr{\frac 12-\frac {\Re(a)}2}}\frac{\zeta\pr{s+a}\zeta\pr{1-s}}{ h^{s+a}k^{1-s}}\frac{ds}{1-s}+\mbox{}\\
&+\frac1{2\pi i}\int_{\pr{\frac 12-\frac {\Re(a)}2}}\frac{\zeta\pr{s+a}\zeta\pr{1-s}}{ h^{s+a}k^{1-s}}\frac{ds}{s+a}\\
=&\,I_a\pr{\frac hk}+I_a\pr{\frac kh},
}
where
\est{
I_a\pr{\frac hk}:=\frac1{2\pi i}\int_{\pr{\frac 12-\frac {\Re(a)}2}}\frac{\zeta\pr{s+a}\zeta\pr{1-s}}{ h^{s+a}k^{1-s}}\frac{ds}{1-s}.
}
The integral is not absolutely convergent, so some care is needed. One could introduce a convergence factor $e^{\delta s^2}$ and let $\delta\rightarrow0^+$ at the end of the argument, or one could work with the understanding that the integrals are to be interpreted as $\lim_{T\rightarrow\infty}\int_{c-iT}^{c+iT}$. We opt for the latter. Recall that $\zeta(s)=\chi(s)\zeta(1-s),$ where
\est{
\chi(1-s)=\pr{(2\pi i )^{-s}+(-2\pi i )^{-s}}\Gamma(s).
}
This leads to
\est{
\frac{1}{2\pi i}\int_{\pr{2}}\frac{\chi(1-s)}{1-s}u^{-s}\,ds&=\frac{-1}{2\pi i}\int_{\pr{2}}\pr{(2\pi i )^{-s}+(-2\pi i )^{-s}}\frac{\Gamma(s)}{s-1}u^{-s}\,d s\\
&=\frac{-1}{2\pi i u}\int_{\pr{1}}\pr{(2\pi i )^{-s-1}+(-2\pi i )^{-s-1}}\Gamma(s)u^{-s}\,d s\\
&=\frac{\sin2\pi u}{\pi u}.
}
Using Cauchy's theorem, the functional equation for $\zeta(s)$, and the Dirichlet series for $\zeta(s+a)\zeta(s)$, we have
\est{
I_a\pr{\frac hk}=&-\tn{Res}_{s=1}\frac{\chi(1-s)\zeta(s+a)\zeta(s)}{h^{s+a}k^{1-s}(1-s)}-\tn{Res}_{s=1-a}\frac{\chi(1-s)\zeta(s+a)\zeta(s)}{h^{s+a}k^{1-s}(1-s)}+\\
&+\frac1{\pi h^{1+a}}\sum_{n=1}^{\infty}\frac{\sigma_{-a}(n)\sin 2\pi n\frac hk}{n}\\
=&-\frac{\zeta\pr{1+a}}{ 2h^{1+a}}+\frac{\zeta\pr{a}}{ ahk^{a}}+\frac1{\pi h^{1+a}}\sum_{n=1}^{\infty}\frac{\sigma_{-a}(n)\sin 2\pi n\frac hk}{n}.
}
By the functional equation for $D$ we see that
\est{
&\frac{D\pr{s,-a,\frac hk}-D\pr{s,-a,-\frac hk}}{2i}=\frac{2}{k}\pr{\frac{k}{2\pi}}^{2-2s-a}\Gamma\pr{1-s-a}\Gamma\pr{1-s}\times\\
&\qquad\times\pr{\cos\pr{\frac\pi2\pr{2s+a}}+\cos\frac{\pi a}2}\pr{D\pr{1-s,a,\frac {\overline h}k}-D\pr{1-s,a,-\frac {\overline h}k}},
}
so that, defining
\est{
S\pr{s,-a,\frac hk}:=& \sum_{n=1}^{\infty}\frac{\sigma_{-a}(n)\sin2\pi n\frac hk}{n^s},
}
we have
\es{\label{sfe}
S\pr{s,-a,\frac hk}=&\,\frac{2}{k}\pr{\frac{k}{2\pi}}^{2-2s-a}\Gamma\pr{1-s-a}\Gamma\pr{1-s}\times\\
&\times\pr{\cos\pr{\frac\pi2\pr{2s+a}}+\cos\frac{\pi a}2}S\pr{1-s,a,\frac{\overline h}k}.
}
In particular, $S\pr{s,-a,\frac hk}$ is regular at $s=1$. Noting that
\est{
\lim_{s\rightarrow1}\Gamma(1-s-a)\Gamma(1-s)\pr{\cos\pr{\frac\pi2\pr{2s+a}}+\cos\frac{\pi a}2}=-\pi\Gamma(-a)\sin\frac{\pi a}2
}
and
\est{
S\pr{0,a,\frac {\overline h}k}=\frac12 c_{a}\pr{\frac {\overline h}k},
}
we obtain, by letting $s\rightarrow1$ in~\eqref{sfe}, the identity
\est{
S\pr{1,-a,\frac hk}=2^{a}\pr{\frac{\pi}{k}}^{1+a}\Gamma(-a)\sin\frac{\pi a}2c_{a}\pr{\frac {\overline h}k},
}
whence
\est{
\sum_{n=1}^{\infty}\frac{\sigma_{-a}(n)\sin2\pi n\frac hk}{\pi nh^{1+a} }=-\pr{\frac{1}{hk}}^{1+a}(2\pi)^{a}\Gamma(-a)\sin\frac{\pi a}2c_{a}\pr{\frac {\overline h}k}.
}
Thus,
\est{
I_a\pr{\frac hk}=-\frac{\zeta\pr{1+a}}{ 2h^{1+a}}+\frac{\zeta\pr{a}}{ ahk^{a}}-\pr{\frac{1}{hk}}^{1+a}(2\pi)^{a}\Gamma(-a)\sin\frac{\pi a}2c_{a}\pr{\frac {\overline h}k}
}
and the theorem follows.
\end{proof}

\appendix


\begin{thebibliography}{0}

\bibitem{Bag}
Bagchi, Bhaskar. \emph{On Nyman, Beurling and Baez-Duarte's Hilbert space reformulation of the Riemann hypothesis.} Proc. Indian Acad. Sci. Math. 116 (2006), no. 2, 137-146; arxiv math.NT/0607733.

\bibitem{Br}
Bruggeman, R.W. \emph{Quantum Maass forms.} The Conference on L-Functions, World Sci. Publ. (2007), 1-15.

\bibitem{BDBLS}
Ba\'ez-Duarte, Luis; Balazard, Michel; Landreau, Bernard; Saias, Eric. \emph{\'Etude de l'autocorrlation multiplicative de la fonction `partie fractionnaire'.} (French) [Study of the multiplicative autocorrelation of the fractional part function] Ramanujan J. 9 (2005), no. 1-2, 215-240; arxiv math.NT/0306251.

\bibitem{BCHB}
Balasubramanian, R.; Conrey, J.B.; Heath-Brown, D.R. \emph{Asymptotic mean square of the product of the Riemann zeta-function and a Dirichlet polynomial.} J. Reine Angew. Math. 357 (1985), 161-181.

\bibitem{BC}
Bettin S.; Conrey, J.B. \emph{A reciprocity formula for a cotangent sum.} Preprint.

\bibitem{C}
Conrey, J.B. \emph{More than two fifths of the zeros of the Riemann zeta function are on the critical line.} J. Reine Angew. Math. 399 (1989), 1-26.

\bibitem{CGh}
Conrey, J.B.; Ghosh, A. \emph{A conjecture for the sixth power moment of the Riemann zeta-function.} Internat. Math. Res. Notices. 15 (1998), 775-780.

\bibitem{CGo}
Conrey J.B.; Gonek, S.M. \emph{High moments of the Riemann zeta-function}.  Duke Math. J.  107 (2001),  no. 3, 577-604.

\bibitem{HL}
Hardy, G.H.; Littlewood, J.E. \emph{Contributions to the theory of the Riemann zeta-function and the theory of the distribution of primes.} Acta Mathematica 41 (1918), 119-196.

\bibitem{In}
Ingham, A.E. \emph{Mean-values theorems in the theory of the Riemann zeta-function.} Proc. Lond. Math. Soc., 27 (1926), 273-300.

\bibitem{Iw}
Iwaniec, Henrik. \emph{On mean values for Dirichlet's polynomials and the Riemann zeta function.} J. London Math. Soc. (2) 22 (1980), no. 1, 39-45.

\bibitem{KS}
Keating, J.P.; Snaith, N.C. \emph{Random matrix theory and $\zeta\pr{\frac12 + it} $.} Comm. in Math. Phys. 214 (2000), 57-89.

\bibitem{LR}
Landreau, Bernard; Richard, Florent. \emph{Le crit\`ere de Beurling et Nyman pour l'hypoth\`ese de Riemann: aspects num\'eriques.} (French) [The Beurling-Nyman criterion for the Riemann hypothesis: numerical aspects] Experiment. Math. 11 (2002), no.3, 349-360.

\bibitem{LZ}
Lewis, J.B.; Zagier, D. \emph{Period functions for Maass wave forms. I.}  Ann. Math. (2), 153 (2001), no. 1, 191-258; arxiv math.NT/0101270.

\bibitem{MOT}
Motohashi, Y. \emph{Spectral theory of the Riemann zeta-function}. Cambridge Tracts in Mathematics, 127, Cambridge University Press, 1997.

\bibitem{Zag}
Zagier, D. \emph{Quantum modular forms}, Clay Math Proceedings AMS 11(11),
659-675, 2010.

\end{thebibliography}
\end{document}